\documentclass[reqno, 12pt]{amsart}
\usepackage{style}

\begin{document}

\title[Optimal Thresholds for Monotone Non-Boolean Functions]{Optimal Thresholds for Monotone Non-Boolean Functions}

\author{Saba Lepsveridze} \address{Massachusetts Institute of Technology \\ 77 Massachusetts Ave \\ Cambridge, MA 02139-4301} \email{sabal@mit.edu}

\author{Allen Lin} \address{Massachusetts Institute of Technology \\ 77 Massachusetts Ave \\ Cambridge, MA 02139-4301} \email{allenees@mit.edu}

\begin{abstract}
    Let $[q] = \{0,1,\ldots,q-1\}$, let $\Delta[q]$ denote the simplex of probability measures on $[q]$, and let $\gamma$ denote the Lebesgue measure normalized on $\Delta[q]$. We prove that for any symmetric monotone function $f \colon[q]^n \to [q]$ and any $a \in [q]$ we have
    \begin{equation*}
        \gamma(\{\mu \in \Delta[q]\;\vert\;\mathbb{P}_{x\sim\mu^{\otimes n}}[f(x)=a] \in (\varepsilon,1-\varepsilon)\}) = O(1/\log n)\text{.}
    \end{equation*}
    We also show that this bound is tight. This improves Kalai and Mossel's previous bound of $O(\log \log n/\log n)$ and answers their question completely.
\end{abstract}

\maketitle

\section{Introduction}

It is well known that monotone symmetric Boolean functions undergo sharp thresholds. This sharp threshold property is prevalent and has numerous applications to random graph theory, social choice theory, and percolation theory. For example, let $\pi_p$ be the distribution on $\{0,1\}$ where $1$ is chosen with probability $p$. Friedgut and Kalai \cite{FriedgutKalai1996} proved that for every $0 < \varepsilon < 1/2$, there exists a constant $C(\varepsilon)$ such that for all $n \geq 2$ and all monotone and symmetric functions $f \colon \{0,1\}^n \to \{0,1\}$, if $\mathbb{E}_{x \sim \pi_p^{\otimes n}}[f(x)] = \varepsilon$ and $\mathbb{E}_{x \sim \pi_q^{\otimes n}}[f(x)] = 1 - \varepsilon$ for some $q > p$, then
\begin{equation*}
    q - p < \frac{C(\varepsilon)}{\log n}\text{.}
\end{equation*}
This result implies that all monotone graph properties undergo sharp thresholds.

Naturally, one might consider functions with more generalized domains and ask if they also undergo sharp thresholds. Let $A$ be a finite set. Here, we consider functions $f \colon A^n \to A$ that are monotone and symmetric. We give the definitions below.

\subsection{Notation}

Following Kalai and Mossel \cite{KalaiMossel2015}, let $S(n)$ denote the group of permutations on elements of the set $\{1,\ldots,n\}$. For $\sigma \in S(n)$ and $x \in A^n$, let $y = x_\sigma$ denote the vector satisfying $y_i = x_{\sigma(i)}$ for all $i \in \{1, \ldots, n\}$. In other words, $y$ is the vector formed by permuting the bits of $x$ according to $\sigma$. 

Moreover, for $a \in A$ and $x, y \in A^n$, we define a partial order $\leq_a$ on $A^n$ as follows. We write $x \leq_a y$ if and only if $\{i \;\vert\;x_i = a\} \subseteq \{i \;\vert\; y_i =a \}$ and for all $i \in \{1, \ldots, n\}$ such that $y_i \neq a$ we have $x_i = y_i$. That is, if $x \leq_a y$, then $y$ can be obtained from $x$ by changing some bits in $x$ to $a$.

Finally, throughout the paper $\log n$ denotes $\log_2 n$. We use the following definitions from Kalai and Mossel \cite{KalaiMossel2015}.

\begin{definition}
    A function $f \colon A^n \to A$ is \emph{monotone} if for all $a \in A$ and $x,y \in A^n$ such that $x \leq_a y$, $f(x) = a$ implies $f(y) = a$.
\end{definition}

\begin{definition}\label{def:symmetric}
    A function $f \colon A^n \to A$ is \emph{symmetric} if there exists some transitive group $\Sigma \subseteq S(n)$ such that $f(x_\sigma) = f(x)$ for all $x \in A^n$ and $\sigma \in \Sigma$.
\end{definition}

We note that other authors refer to such functions described in Definition \ref{def:symmetric} as \emph{transitive symmetric} and define $f \colon A^n \to A$ to be \emph{symmetric} if $f(x_\sigma) = f(x)$ for all $x \in A^n$ and all $\sigma \in S(n)$.

\subsection{Main Result}

Let $\Delta[A]$ denote the simplex of probability measures on $A$, and let $\gamma$ denote the Lebesgue measure normalized on $\Delta[A]$. Let $f \colon A^n \to A$ be some symmetric and monotone function. Kalai and Mossel \cite{KalaiMossel2015} proved that there exists some constant $C=C(|A|)$ such that
\begin{equation*}
    \gamma(\{\mu \in \Delta[A]\;\vert\;\varepsilon \leq \mathbb{P}_{x\sim\mu^{\otimes n}}[f(x) = a] \leq 1 - \varepsilon\}) \leq C(\log(1-\varepsilon)-\log(\varepsilon))\frac{\log \log n}{\log n}\text{,}
\end{equation*}
leaving a gap of $\log\log n$ compared to the Boolean case when $|A| = 2$. In their paper, Kalai and Mossel \cite{KalaiMossel2015} conjectured that this can be improved to $O(1/\log n)$ as in the Boolean case. In this paper, we prove this conjecture.

\begin{theorem}\label{thm:main}
    There exists a constant $C = C(|A|)$ such that if $f \colon A^n \to A$ is symmetric and monotone, then
    \begin{equation}\label{eqn:main}
        \gamma(\{\mu \in \Delta[A]\;\vert\;\varepsilon \leq \mathbb{P}_{x\sim\mu^{\otimes n}}[f(x) = a] \leq 1 - \varepsilon\}) \leq \frac{C(\log(1-\varepsilon)-\log(\varepsilon))}{\log n}
    \end{equation}
    for any $a \in A$ and any $0 < \varepsilon < 1/2$.
\end{theorem}
Furthermore, for any $\mu \in \Delta[A]$, let $E(\mu)$ be the event that there exists some $a \in A$ such that $\mathbb{P}_{x\sim\mu^{\otimes n}}[f(x)=a] \in (\varepsilon,1-\varepsilon)$. Note that by the union bound, we also have
\begin{equation*}
    \gamma(\{\mu \in \Delta[A]\;\vert\;E(\mu)\;\text{is true}\}) \leq \frac{C(\log(1-\varepsilon)-\log(\varepsilon))}{\log n}\text{.}
\end{equation*}
As such, all bounds of applications present in Kalai and Mossel \cite{KalaiMossel2015} such as monotone graph properties and Condorcet's Jury Theorem are improved to $O(1/\log n)$.

\subsection*{Acknowledgements} The authors would like to thank Elchanan Mossel for suggesting this problem and for the many helpful discussions. Furthermore, the authors would like to thank the referees for their invaluable feedback and comments, which greatly improved the quality of the paper. S.L. was partially supported by the NSF–Simons Research Collaboration Grant (Award No.~2031883).

\section{Previous Sharp Threshold Results}

\subsection{Friedgut and Kalai's Result for \texorpdfstring{$|A| = 2$}{q = 2}}

Friedgut and Kalai \cite{FriedgutKalai1996} proved the following sharp threshold phenomenon for $|A| = 2$. Here let $A = \{0,1\}$. 
\begin{theorem}[Friedgut and Kalai \cite{FriedgutKalai1996}]\label{thm:friedgut-kalai}
    There exists a constant $C$ such that if $f \colon \{0,1\}^n \to \{0,1\}$ is symmetric and monotone, then
    \begin{equation*}
        \gamma(\{\mu \in \Delta[A]\;\vert\;\varepsilon \leq \mathbb{E}_{x\sim \mu^{\otimes n}}[f(x)] \leq 1-\varepsilon\}) \leq \frac{C\log(1/\varepsilon)}{\log n}
    \end{equation*}
    for any $0 < \varepsilon < 1/2$.
\end{theorem}

Fundamental to the proofs of these threshold results is the notion of influences. Influences of variables of Boolean functions have been extensively studied and have led to applications in combinatorics, theoretical computer science, and other areas. There are various definitions for influences, all based on dividing the domain into one-dimensional subspaces called \emph{fibres}. For the sake of discussion, we consider general product probability spaces $X = X_1 \times \cdots \times X_n$. 

\begin{definition}
    For any $x = (x_1, \ldots, x_n) \in X$ and any $1 \leq k \leq n$, the \emph{fibre} of $x$ in the $k$th direction is
    \begin{equation*}
        s_k(x) = \{y \in X \;\vert\; y_i = x_i \;\text{for all}\;i \neq k\}\text{.}
    \end{equation*}
    For a function $f \colon X \to \{0,1\}$, the restriction of $f$ to $s_k(x)$ is denoted $f_k^x \colon X_k \to \{0,1\}$ and is defined by
    \begin{equation*}
        f_k^x(t) = f(x_1, \ldots, x_{k-1},t,x_{k+1}, \ldots, x_n)\text{.}
    \end{equation*}
\end{definition}

The original definition of influence in general product spaces was introduced by Bourgain, Kahn, Kalai, Katznelson, and Linial \cite{BKKKL} as follows.

\begin{definition}\label{def:BKKKL-influence}
    For any function $f \colon X \to \{0,1\}$ and any $1 \leq k \leq n$, the \emph{influence} of the $k$th variable on $f$ is
    \begin{equation*}
        I_k(f) = \mathbb{P}_{x \in X}[f_k^x\;\text{is not constant}]\text{.}
    \end{equation*}
    In particular, if $f \colon \{0,1\}^n \to \{0,1\}$ is endowed with the product measure $\mu^{\otimes n}$, then
    \begin{equation*}
        I_k\left(f;\mu^{\otimes n}\right) = \mathbb{P}_{x \sim \mu^{\otimes n}}[f_k^x\;\text{is not constant}]\text{.}
    \end{equation*}
\end{definition}

The most well-known theorem regarding influences in product spaces is the BKKKL theorem \cite{BKKKL}.
\begin{theorem}[Bourgain, Kahn, Kalai, Katznelson, and Linial \cite{BKKKL}]\label{thm:BKKKL}
    There exists a universal constant $c$ such that for any function $f \colon X \to \{0,1\}$, there exists a coordinate $k$ such that
    \begin{equation*}
        I_k(f) \geq c\Var(f)\log n/n\text{.}
    \end{equation*}
\end{theorem}

Theorem \ref{thm:friedgut-kalai} follows swiftly by Theorem \ref{thm:BKKKL} and the following lemma of Russo and Margulis that relates the change of $\mathbb{E}_{x\sim\pi_p^{\otimes n}}[f]$ to the influence of $f$.
\begin{lemma}[Russo \cite{Russo1978}, Margulis \cite{Margulis1974}]\label{lemma:russo-margulis}
    Let $f \colon \{0,1\}^n \to \{0,1\}$ be a monotone function. Then under the measure $\pi_p^{\otimes n}$, we have
    \begin{equation*}
        \frac{d\mathbb{E}_{x\sim\pi_p^{\otimes n}}[f(x)]}{dp} = \sum_{k=1}^{n}{I_k(f;\pi_p^{\otimes n})}\text{.}
    \end{equation*}
\end{lemma}
Note that if $f\colon\{0,1\}^n \to \{0,1\}$ is symmetric, then all the coordinates have equal influence. Since $f$ is also monotone, by Theorem \ref{thm:BKKKL} and Lemma \ref{lemma:russo-margulis} we obtain
\begin{equation*}
    \frac{d\mathbb{E}_{x\sim\pi_p^{\otimes n}}[f(x)]}{dp} = \sum_{k=1}^{n}{I_k(f;\pi_p^{\otimes n})} \geq c\Var_{\pi_p^{\otimes n}}(f)\log n\text{.}
\end{equation*}
This implies that $\mathbb{E}_{x\sim\pi_p^{\otimes n}}[f(x)]$ varies quickly as a function of $p$. In fact, $\mathbb{E}_{x\sim\pi_{p}^{\otimes n}}[f(x)]$ increases from $\varepsilon$ to $1-\varepsilon$ as $p$ increases on an interval of length $1/\log n$, up to some constant depending on $\varepsilon$.

\subsection{Kalai and Mossel's Result for \texorpdfstring{$|A| > 2$}{q > 2}}

Kalai and Mossel generalized Friedgut and Kalai's results to sets with more than $2$ elements.

\begin{theorem}[Kalai and Mossel \cite{KalaiMossel2015}]\label{thm:kalai-mossel}
    There exists a constant $C = C(|A|)$ such that if $f \colon A^n \to A$ is symmetric and monotone, then
    \begin{equation*}
        \gamma(\{\mu \in \Delta[A]\;\vert\;\varepsilon < \mathbb{P}_{x \sim \mu^{\otimes n}}[f(x)=a] < 1-\varepsilon\}) \leq C(\log(1-\varepsilon)-\log(\varepsilon))\frac{\log\log n}{\log n}
    \end{equation*}
    for any $a \in A$ and any $0 < \varepsilon < 1/2$.
\end{theorem}
Kalai and Mossel \cite{KalaiMossel2015} achieved this using another widely known definition of influence.

\begin{definition}\label{def:variation-influence}
    For any function $f \colon X \to \{0,1\}$ and any $1 \leq k \leq n$, the \emph{influence} of the $k$th variable on $f$ is
    \begin{equation*}
        \widetilde{I}_{k}(f) = \mathbb{E}_{x\in X}[\Var(f_k^x)]\text{.}
    \end{equation*}
    In particular, if $f \colon A^n \to \{0,1\}$ is a function with product measure $\mu^{\otimes n}$ for some $\mu \in \Delta[A]$, then denote
    \begin{equation*}
        \widetilde{I}_k(f;\mu^{\otimes n}) = \mathbb{E}_{x \sim \mu^{\otimes n}}[\Var(f_k^x)]\text{.}
    \end{equation*}
\end{definition}
Using a hypercontractivity result of Wolff \cite{Wolff2007}, Kalai and Mossel generalized the result of Talagrand \cite{Talagrand1994} on the lower bound of the influence of a function.

\begin{theorem}[Kalai and Mossel \cite{KalaiMossel2015}]
    There exists an universal constant $c$ such that for any probability space $(X,\mu)$ and any symmetric function $f \colon X^n \to \{0,1\}$, we have
    \begin{equation*}
        \sum_{k=1}^{n}{\widetilde{I}_k(f;\mu^{\otimes n})} \geq \frac{c}{\log\left(1/\min_{x \in X}{\mu(x)}\right)}\Var_{\mu^{\otimes n}}(f)\log n\text{.}
    \end{equation*}
\end{theorem}

Kalai and Mossel \cite{KalaiMossel2015} also proved a generalization of the Russo-Margulis lemma under a certain class of monotone functions (see Definition \ref{def:a-monotone}). Let $[q] \coloneqq \{0,1,\ldots,q-1\}$.
\begin{lemma}[Kalai and Mossel \cite{KalaiMossel2015}]
    Let $f \colon [q]^n \to \{0,1\}$ be a $0$-monotone function, and let $\mu \in \Delta[q]$. Write $\mu = (1-\mu(0))\mu' + \mu(0)\delta_0$ for some $\mu' \in \Delta[q]$ such that $\mu'(0)=0$, and let $\mu_t = \mu + t\delta_0 - t\mu'$ for $t \in [0,1-\mu(0)]$. Then
    \begin{equation*}
        \frac{d\mathbb{E}_{x\sim\mu_t^{\otimes n}}[f(x)]}{dt} \geq \sum_{k=1}^{n}{\widetilde{I}_k(f;\mu^{\otimes n})}\text{.}
    \end{equation*}
\end{lemma}
Their proof proceeds in the same manner as that of Friedgut and Kalai. In particular, if $f \colon [q]^n \to \{0,1\}$ is both symmetric and $0$-monotone under $\mu^{\otimes n}$, then
\begin{equation*}
    \frac{d\mathbb{E}_{x\sim\mu_t^{\otimes n}}[f(x)]}{dt} \geq \frac{c}{\log\left(1/\min_{x\in X}{\mu(x)}\right)}\Var_{\mu^{\otimes n}}(f)\log n\text{.}
\end{equation*}
Thus, $\mathbb{E}_{x \sim \mu_t^{\otimes n}}[f(x)]$ varies quickly as a function of $t$ with a dependence on $\min_{i \in \{1,\ldots,n\}}{\mu(i)}$. This dependence gives the $O(\log \log n/\log n)$ bound.

\subsection{General Influences}

Following Keller \cite{Keller2011}, we use a more general definition of influences in general product spaces.
\begin{definition}\label{def:h-influences}
    Let $h \colon [0,1] \to \mathbb{R}$. For any function $f \colon X \to \{0,1\}$ and any $1 \leq k \leq n$, the \emph{$h$-influence} of the $k$th variable on $f$ is
    \begin{equation*}
        I_k^h(f) = \mathbb{E}_{x \in X}[h(\mathbb{E}[f_k^x])]\text{.}
    \end{equation*}
\end{definition}
Note that Definition \ref{def:BKKKL-influence} is obtained by the function $h(t) = \mathbb{I}[t \in (0,1)]$ and Definition \ref{def:variation-influence} is obtained by the function $h(t) = t(1-t)$. 

Keller \cite{Keller2011} proved a generalization of the BKKKL theorem for $h$-influences.
\begin{theorem}[Keller \cite{Keller2011}]\label{thm:keller-bkkkl}
    Denote by $\Ent(t) = -t\log t - (1-t)\log(1-t)$ the entropy function. Let $h \colon [0,1] \to \mathbb{R}$ be such that $h(t) \geq \Ent(t)$ for all $t \in [0,1]$. Then there exists a universal constant $c$ such that for any function $f \colon [0,1]^n \to \{0,1\}$, there exists a coordinate $k$ such that
    \begin{equation*}
        I_k^h(f) \geq c \Var(f)\log n/n\text{.}
    \end{equation*}
\end{theorem}

\section{Proof of Theorem \ref{thm:main}}

\subsection{Reduction to one value}
Assume $A = [q] \coloneqq \{0,1,\ldots, q-1\}$ without loss of generality. Let $f \colon [q]^n \to [q]$ be a symmetric and monotone function. Fix some $a \in [q]$ and some $\varepsilon \in (0,1/2)$. Define $\widetilde{f} \colon [q]^n \to \{0,1\}$ by
\begin{equation*}
    \widetilde{f}(x) = \mathbb{I}[f(x) = a]\text{.}
\end{equation*}
Note that $\widetilde{f}$ inherits the monotonicity of $f$ in the following manner.
\begin{definition}\label{def:a-monotone}
    We say $f \colon [q]^n \to \{0,1\}$ is \emph{$a$-monotone} if $x \leq_a y$ implies $f(x) \leq f(y)$.
\end{definition}
Indeed, $\widetilde{f}$ is $a$-monotone. Assume $a = 0$. Therefore, to prove \eqref{eqn:main}, it suffices to prove that there exists some universal constant $C = C(q)$ such that if $f \colon [q]^n \to \{0,1\}$ is symmetric and $0$-monotone, then 
\begin{equation}\label{eqn:main-specific}
    \gamma\left(\left\{\mu \in \Delta[q] \;\bigg\vert\; \varepsilon \leq \mathbb{E}_{x\sim \mu^{\otimes n}}\left[f(x)\right] \leq 1-\varepsilon\right\}\right) \leq \frac{C(\log(1-\varepsilon)-\log(\varepsilon))}{\log n}\text{.}
\end{equation}
We sketch a brief outline of the proof. Denote the region under consideration by
\begin{align*}
    \mathscr{D} &\coloneqq \{\mu \in \Delta[q] \;\vert\;\varepsilon \leq \mathbb{E}_{x \sim \mu^{\otimes n}}[f(x)] \leq 1-\varepsilon\}\text{.}
\end{align*}
We consider probability measures in $\Gamma \coloneqq \{\mu \in \Delta[q] \;\vert\;\mu(0)=0\}$, and for each $\mu \in \Gamma$ we write $\mu_t \coloneqq t\delta_0 + (1-t)\mu$ and consider the set of measures $\{\mu_t\;\vert\;t \in [0,1]\}$. Clearly
\begin{equation*}
    \bigcup_{\mu \in \Gamma}{\{\mu_t\;\vert\;t\in[0,1]\}} = \Delta[q]\text{,}
\end{equation*}
so for each measure $\mu \in \Gamma$ we bound the length of the interval $\{t \in [0,1]\;\vert\;\mu_t \in \mathscr{D}\}$. To do this, we show that $\mathbb{E}_{x\sim\mu_t^{\otimes n}}[f(x)]$ varies quickly as a function of $t$ by establishing a lower bound on $d\mathbb{E}_{x\sim\mu_t^{\otimes n}}[f(x)]/dt$. This lower bound on the derivative depends on the second-smallest atom $\alpha = \min_{j \in [q]\setminus\{0\}}\mu(j)$ of $\mu$ (since $\mu(0) = 0$ is the smallest). To this end, we partition $\Gamma$ into $q-1$ regions $R_i = \{\mu \in \Gamma \;\vert\;\mu(i) = \alpha\}$ based on the second-smallest atom and show that the length of the interval $\{t\in[0,1]\;\vert\;\mu_t\in\mathscr{D}\}$ is small for each region, with only a dependence on $\alpha$. Finally, we remove the dependence on $\alpha$ and show that
\begin{equation*}
    \gamma\left(\{\mu_t\;\vert\;\mu \in R_i\;\text{and}\;t\in[0,1]\} \cap \mathscr{D}\right) \leq \frac{C(\ln(1-\varepsilon)-\ln(\varepsilon))}{\log n}
\end{equation*}
for some constant $C$ depending only on $q$. This proves \eqref{eqn:main-specific} since it holds for all regions $R_i$.

\subsection{Generalization of Russo-Margulis}

Here, we show a generalization of the Russo-Margulis lemma. This follows along the same lines as in \cite{KalaiMossel2015}, but we include it for the sake of completeness. Define the region $\Gamma \coloneqq \{\mu \in \Delta[q]\;\vert\;\mu(0) = 0\}$.

\begin{lemma}\label{lemma:russo-margulis-one-variable}
    Let $f \colon [q] \to \{0,1\}$ be a $0$-monotone function. Let $\mu \in \Gamma$ be some probability measure and define $\mu_t = t\delta_0 + (1-t)\mu$ for $t \in [0,1]$. Then
    \begin{align*}
        \frac{d\mathbb{E}_{x \sim \mu_t}[f(x)]}{dt} &=  \frac{\mathbb{I}[f\;\text{is not constant}]\mathbb{E}_{x\sim\mu_t}[1-f(x)]}{1-t}\text{.}
    \end{align*}
\end{lemma}
\begin{proof}
    Note $d\mathbb{E}_{x \sim \mu_t}[f(x)]/dt = f(0) - \mathbb{E}_{x \sim \mu}[f(x)]$ and that this is $0$ when $f$ is constant. If $f$ is not constant, then $f(0) = 1$ since $f$ is $0$-monotone and the expression becomes $\mathbb{E}_{x\sim\mu}[1-f(x)]$. But $\mu = (\mu_t-t\delta_0)/(1-t)$ and $\mathbb{E}_{x\sim \delta_0}[1-f(x)] = 0$ since $f(0) = 1$, so 
    \begin{equation*}
        \mathbb{E}_{x\sim\mu}[1-f(x)] = \frac{\mathbb{E}_{x\sim \mu_t}[1-f(x)]-t\mathbb{E}_{x\sim\delta_0}[1-f(x)]}{1-t} = \frac{\mathbb{E}_{x\sim\mu_t}[1-f(x)]}{1-t}\text{,}
    \end{equation*}
    as required.
\end{proof}

\begin{lemma}\label{lemma:russo-margulis-generalized}
    Let $f \colon [q]^n \to \{0,1\}$ be a $0$-monotone function. Let $\mu \in \Gamma$ be some probability measure and define $\mu_t = t\delta_0 + (1-t)\mu$ for $t \in [0,1]$. Then
    \begin{equation*}
        \frac{d\mathbb{E}_{x\sim \mu_t^{\otimes n}}[f(x)]}{dt} = \frac{1}{1-t}\sum_{k=1}^{n}{\mathbb{E}_{x_{-k}\sim \mu_t^{\otimes (n-1)}}[\mathbb{I}[f_k^x\;\text{is not constant}]\mathbb{E}_{x_k \sim \mu_t}[1-f_k^x(x_k)]]}\text{.}
    \end{equation*}
\end{lemma}
\begin{proof}
    By the product rule and Lemma \ref{lemma:russo-margulis-one-variable},
    \begin{align*}
        \frac{d\mathbb{E}_{x\sim \mu_t^{\otimes n}}[f(x)]}{dt} &= \sum_{x\in[q]^n}{\sum_{k=1}^{n}{\frac{d\mu_t(x_k)}{dt}\prod_{\ell \neq k}{\mu_t(x_\ell)}}f(x)} \\
        &= \sum_{k=1}^{n}{\sum_{x_k \in [q]}\mathbb{E}_{x_{-k}\sim \mu_t^{\otimes (n-1)}}\left[\frac{d\mu_t(x_k)}{dt}f_k^x(x_k)\right]} \\
        &= \sum_{k=1}^{n}{\mathbb{E}_{x_{-k}\sim\mu_t^{\otimes (n-1)}}\left[\frac{d\mathbb{E}_{x_k\sim \mu_t}[f_k^x(x_k)]}{dt}\right]} \\
        &= \frac{1}{1-t}\sum_{k=1}^{n}{\mathbb{E}_{x _{-k}\sim \mu_t^{\otimes (n-1)}}[\mathbb{I}[f_k^x\;\text{is not constant}]\mathbb{E}_{x_k \sim \mu_t}[1-f_k^x(x_k)]]}\text{,}
    \end{align*}
    as required.
\end{proof}

\subsection{A Lower Bound on the Derivative}

Henceforth, we assume that $f$ is not constant and that $q > 2$ since the $q = 2$ case follows from Friedgut and Kalai \cite{FriedgutKalai1996} (see Theorem \ref{thm:friedgut-kalai}). In the sequel, $C$ denotes different constants at different lines, depending on $q$ only.

In this section, we prove the following.
\begin{proposition}\label{prop:lower-bound-on-derivative}
    Let $f \colon [q]^n \to \{0,1\}$ be a $0$-monotone and symmetric function and let $\mu \in \Gamma$ with second-smallest atom $\alpha = \min_{j \in [q]\setminus\{0\}}{\mu(j)}$. Define the measure $\mu_t = t\delta_0 + (1-t)\mu$ for $t \in [0,1]$. There exists an absolute constant $C$ such that
    \begin{equation*}
        \frac{d\mathbb{E}_{x\sim\mu_t^{\otimes n}}[f(x)]}{dt}\geq \frac{C\mathbb{E}_{x\sim\mu_t^{\otimes n}}[f(x)](1-\mathbb{E}_{x\sim\mu_t^{\otimes n}}[f(x)])\log n}{\log(1/\alpha)}\text{.}
    \end{equation*}
\end{proposition}

Before we prove the proposition, we need the following lemma.
\begin{lemma}\label{lemma:bounding-alpha}
    Let $f \colon [q]^n \to \{0,1\}$ be a $0$-monotone function, let $\mu \in \Gamma$ with second-smallest atom $\alpha = \min_{j \in [q]\setminus\{0\}}{\mu(j)}$, and let $\mu_t = t\delta_0 + (1-t)\mu$ for $t \in [0,1]$. If on any fibre $s_k(x)$ the function $f_k^x \colon [q]\to\{0,1\}$ is not constant, then
    \begin{equation*}
        \alpha(1-t) \leq \mathbb{E}_{x_k\sim\mu_t}[1-f_k^x(x_k)]\text{.}
    \end{equation*}
\end{lemma}
\begin{proof}
    Fix $\mu \in \Gamma$ and any fibre $s_k(x)$. Note that $f_k^x(0) = 1$ by $0$-monotonicity and $f_k^x(j) = 0$ for some $j \in [q] \setminus \{0\}$ since $f_k^x$ is not constant. Therefore,
    \begin{equation*}
        \mathbb{E}_{x_k \sim \mu_t}[1-f_k^x(x_k)] \geq \mu_t(j)(1-f_k^x(j)) \geq \alpha(1-t)\text{,}
    \end{equation*}
    as required.
\end{proof}

Now we prove Proposition \ref{prop:lower-bound-on-derivative}.
\begin{proof}[Proof of Proposition \ref{prop:lower-bound-on-derivative}]
    Let $\alpha_t = \alpha(1-t)$. By Lemma \ref{lemma:bounding-alpha} it follows that
    \begin{equation}\label{eqn:bounding-log-1/alpha}
        \log\left(\frac{1}{\alpha_t}\right) \geq \log\left(\frac{1}{\mathbb{E}_{x_k\sim\mu_t}[1-f_k^x(x_k)]}\right)\text{.}
    \end{equation}
    Let
    \begin{equation*}
        \Phi_k(\mu_t) = \mathbb{E}_{x_{-k}\sim\mu_t^{\otimes (n-1)}}[\mathbb{I}[f_k^x\;\text{is not constant}]\mathbb{E}_{x_k\sim\mu_t}[1-f_k^x(x_k)]]\text{.}
    \end{equation*}
    Then by \eqref{eqn:bounding-log-1/alpha} we obtain
    \begin{align*}
        \left(2+2\log\left(\frac{1}{\alpha_t}\right)\right)\Phi_k(\mu_t) &= \mathbb{E}_{x_{-k}\sim\mu_t^{\otimes (n-1)}}\bigg[2\mathbb{I}[f_k^x\;\text{is not constant}] \bigg(\mathbb{E}_{x_k\sim\mu_t}[1-f_k^x(x_k)] \\
        &\quad+\log\left(\frac{1}{\alpha_t}\right)\mathbb{E}_{x_k\sim\mu_t}[1-f_k^x(x_k)]\bigg)\bigg] \\
        &\geq \mathbb{E}_{x_{-k}\sim\mu_t^{\otimes (n-1)}}\bigg[2\mathbb{I}[f_k^x\;\text{is not constant}] \bigg(\mathbb{E}_{x_k\sim\mu_t}[1-f_k^x(x_k)] \\
        &\quad+\log\left(\frac{1}{\mathbb{E}_{x_k\sim\mu_t}[1-f_k^x(x_k)]}\right)\mathbb{E}_{x_k\sim\mu_t}[1-f_k^x(x_k)]\bigg)\bigg]\text{.}
    \end{align*}
    Under the measure $\mu_t^{\otimes n}$ we naturally transform $f$ to a function $F \colon [0,1]^n \to \{0,1\}$ as follows. Define $G \colon [0,1] \to [q]$ by
    \begin{equation*}
        G(x) = i\quad\text{if}\;x\in\left[\sum_{\ell < i}{\mu_t(\ell)},\sum_{\ell \leq i}{\mu_t(\ell)}\right) \qquad \text{and} \qquad G(1) = q-1\text{.}
    \end{equation*}
    Then define $F \colon [0,1]^n \to \{0,1\}$ by
    \begin{equation*}
        F(x_1,\ldots,x_n) = f(G(x_1),\ldots,G(x_n))\text{.}
    \end{equation*}
    It follows that
    \begin{align*}
        \left(2+2\log\left(\frac{1}{\alpha_t}\right)\right)\Phi_k(\mu_t) &\geq \mathbb{E}_{x_{-k}\in[0,1]^{n-1}}\bigg[2\mathbb{I}[F_k^x\;\text{is not constant}]\bigg(\mathbb{E}_{x_k \in [0,1]}[1-F_k^x(x_k)] \\
        &\quad +\log\left(\frac{1}{\mathbb{E}_{x_k\in[0,1]}[1-F_k^x(x_k)]}\right)\mathbb{E}_{x_k\in[0,1]}[1-F_k^x(x_k)]\bigg)\bigg]\text{.}
    \end{align*}
    Define the function $h \colon [0,1] \to \mathbb{R}$ by
    \begin{equation*}
        h(t) = 2\mathbb{I}[t \in (0,1)]\left(1-t\right)(1-\log(1-t))\text{.}
    \end{equation*}
    We verify that $h(t)\ge \Ent(t)$ for all $t \in [0,1]$. Indeed, for $t\in\{0,1\}$ both sides are zero. For $t\in(0,1)$, note that $\ln(1+\delta) \leq \delta$ for all $\delta \geq 0$. Take $\delta = 1/t-1$ gives $\ln(1/t) \leq 1/t-1$, so $t\ln(1/t) \leq 1-t$. Therefore, $-t\log t \le (\log e)(1-t)$, and
    \begin{equation*}
        \Ent(t) \leq (\log e)(1-t) -(1 -t)\log(1-t) \leq 2(1-t)(1-\log(1-t)) = h(t)\text{.}
    \end{equation*}

    By Theorem \ref{thm:keller-bkkkl} and the symmetry of $F$ there exists some universal constant $C$ such that
    \begin{align*}
        \left(2+2\log\left(\frac{1}{\alpha_t}\right)\right)\frac{d\mathbb{E}_{x\sim\mu_t^{\otimes n}}[f(x)]}{dt} &= \frac{1}{1-t}\sum_{k=1}^{n}{\left(2+2\log\left(\frac{1}{\alpha_t}\right)\right)\Phi_k(\mu_t)} \\
        &\geq \frac{1}{1-t}\sum_{k=1}^{n}{I_k^h(F)} \\
        &\geq \frac{C\Var(F)\log n}{1-t} \\
        &= \frac{C\Var_{\mu_t^{\otimes n}}(f)\log n}{1-t}\text{.}
    \end{align*}
    As such,
    \begin{align*}
        \frac{d\mathbb{E}_{x\sim\mu_t^{\otimes n}}[f(x)]}{dt} &\geq \frac{C\mathbb{E}_{x\sim\mu_t^{\otimes n}}[f(x)](1-\mathbb{E}_{x\sim\mu_t^{\otimes n}}[f(x)])\log n}{(1-t)(1+\log(1/\alpha_t))} \\&\geq \frac{C\mathbb{E}_{x\sim\mu_t^{\otimes n}}[f(x)](1-\mathbb{E}_{x\sim\mu_t^{\otimes n}}[f(x)])\log n}{\log(1/\alpha)}\text{,}
    \end{align*}
    where the last inequality follows from
    \begin{equation*}
        (1-t)(1+\log(1/\alpha_t)) \leq 4\log(1/\alpha)\text{.}
    \end{equation*}
    Indeed, since $\alpha \leq 1/2$ we have $1 \leq \log(1/\alpha_t)$, so
    \begin{align*}
        (1-t)(1+\log(1/\alpha_t)) \leq 2(1-t)\log(1/\alpha_t) = 2(1-t)\log(1/\alpha) + 2(1-t)\log(1/(1-t))\text{.}
    \end{align*}
    The result follows since the first summand is at most $2\log(1/\alpha)$, and the second summand is at most $2\Ent(t) \leq 2 \leq 2\log(1/\alpha)$.
\end{proof}

\subsection{Line Segments Have Short Length}

Proposition \ref{prop:lower-bound-on-derivative} shows that the derivative $d\mathbb{E}_{x\sim\mu_t^{\otimes n}}[f(x)]/dt$ is large with a dependence on the second-smallest atom $\alpha$ of $\mu$. We partition $\Gamma$ into $q-1$ regions
\begin{equation*}
    R_i = \left\{\mu\in \Gamma\;\bigg\vert\;\mu(i) = \min_{j \in [q]\setminus\{0\}}{\mu(j)}\right\}
\end{equation*}
for $i \in \{1,\ldots,q-1\}$. In other words, $R_i$ is the set of all measures $\mu$ in $\Gamma$ such that the second-smallest atom of $\mu$ is $\mu(i)$. Fix a region $R_i$ and suppose that $\mu \in R_i$. Then $\alpha = \mu(i)$ by definition. Recall that
\begin{equation*}
    \mathscr{D} \coloneqq \{\mu \in \Delta[q] \;\vert\;\varepsilon \leq \mathbb{E}_{x\sim \mu^{\otimes n}}[f(x)]\leq 1-\varepsilon\}\text{.}
\end{equation*}
In this section we show the following. Again, $C$ denotes different constants at different lines, depending on $q$ only. Here, $\len(I)$ denotes the length of an interval $I$.
\begin{proposition}\label{prop:length-of-line}
    Let $f \colon [q]^n \to \{0,1\}$ be a $0$-monotone and symmetric function, and let $\mu \in R_i$. Define $\mu_t = t\delta_0 + (1-t)\mu$ for $t \in [0,1]$. There exists a universal constant $C$ such that
    \begin{equation*}
        \len(\{t\in [0,1]\;\vert\;\mu_t \in \mathscr{D}\})\leq \frac{C(\ln(1/2)-\ln(\varepsilon))\log(1/\mu(i))}{\log n}
    \end{equation*}
\end{proposition}
\begin{proof}
    Fix $\mu \in R_i$. By Proposition \ref{prop:lower-bound-on-derivative}, there exists a constant $C$ such that
    \begin{equation}\label{eqn:lower-bound-on-derivative}
        \frac{d\mathbb{E}_{x\sim\mu_t^{\otimes n}}[f(x)]}{dt}\geq \frac{C\mathbb{E}_{x\sim\mu_t^{\otimes n}}[f(x)](1-\mathbb{E}_{x\sim\mu_t^{\otimes n}}[f(x)])\log n}{\log(1/\mu(i))}\text{.}
    \end{equation}
    If $\{t\in[0,1]\;\vert\;\mu_t \in \mathscr{D}\} = \varnothing$, then there is nothing to prove. Otherwise, by monotonicity, there exists $s \in [0,1]$ such that $\mathbb{E}_{x\sim\mu_s^{\otimes n}}[f(x)] = 1 - \varepsilon$. Let $r = \inf\left\{t\in[0,1]\;\Big\vert\;\mathbb{E}_{x\sim\mu_t^{\otimes n}}[f(x)] \geq 1/2\right\}$. Observe that for all $t \geq r$, we have $\mathbb{E}_{x\sim\mu_t^{\otimes n}}[f(x)] \geq 1/2$, so \eqref{eqn:lower-bound-on-derivative} implies
    \begin{equation}\label{eqn:log-derivative}
        \frac{d}{dt} \ln\left[\mathbb{E}_{x\sim\mu_t^{\otimes n}}[1-f(x)]\right] \leq  -\frac{C\log n}{\log(1/\mu(i))}\text{.}
    \end{equation}
    Integrating, we get that  
    \begin{equation*}
        \ln(\varepsilon) - \ln(1/2) \leq  \int_{r}^{s} \frac{d}{dt}  \ln\left[\mathbb{E}_{x\sim\mu_t^{\otimes n}}[1-f(x)]\right] \,dt \leq -\frac{C\log n}{\log(1/\mu(i))}(s-r)\text{,}
    \end{equation*}
    which implies
    \begin{equation*}
        s-r \leq C \, \frac{\ln(1/2)-\ln(\varepsilon)}{\log n}\log(1/\mu(i))\text{.}
    \end{equation*}
    Similarly, let $p = \inf\left\{t\in[0,1]\;\Big\vert\;\mathbb{E}_{x\sim\mu_t^{\otimes n}}[f(x)] \geq \varepsilon\right\}$. If $p = r$, then there is nothing more to prove. Otherwise, for all $t \in [p, r)$, we have $\mathbb{E}_{x\sim\mu_t^{\otimes n}}[f(x)] < 1/2$, so \eqref{eqn:lower-bound-on-derivative} implies
    \begin{equation*}
        \ln(1/2)-\ln(\varepsilon) \geq \int_{p}^{r}
        \frac{d}{dt}\ln\left[\mathbb{E}_{x\sim\mu_t^{\otimes n}}[f(x)]\right] \,dt \geq C \, \frac{\log n}{\log(1/\mu(i))}(r-p)\text{.}
    \end{equation*}
    This implies that
    \begin{equation*}
        r-p \leq C\, \frac{\ln(1/2)-\ln(\varepsilon)}{\log n}\log(1/\mu(i))\text{.}
    \end{equation*}
    Therefore,
    \begin{equation*}
        s-p \leq C\, \frac{\ln(1/2)-\ln(\varepsilon)}{\log n}\log(1/\mu(i))\text{,}
    \end{equation*}
    which proves the proposition.
\end{proof}

\subsection{Cross Sections Have Small Lebesgue Measure}

Proposition~\ref{prop:length-of-line} shows that, for each fixed probability measure
$\mu \in R_i$, the set of parameters $t\in[0,1]$ for which
$\mu_t \in \mathscr{D}$ has length
\begin{equation*}
    O\left(\frac{\log(1/\mu(i))}{\log n}\right)\text{.}
\end{equation*}
In the next proposition, we show that after integrating over all choices of
$\mu\in R_i$, the total contribution of the family $\{\mu_t \;|\; \mu\in R_i\;\text{and}\; t\in[0,1]\}$ inside $\mathscr{D}$ is only $O(1/\log n)$.

To do this, define the central measure $\mu^{*}\in\Gamma$ by $\mu^{*}(i)={1}/{(q-1)}$ for all $i \in [q] \setminus \{0\},$ and define the boundary face
\begin{equation*}
    \Gamma_i \coloneqq \{\mu\in R_i \;\vert\; \mu(i)=0\}\text{.}
\end{equation*}
For example, Figure~\ref{fig:q=4} depicts $\Gamma$ when $q = 4$, so $\mu^{*}(1)=\mu^{*}(2)=\mu^{*}(3)=1/3$ and $\mu^{*}(0)=0$. 

For each $\eta\in\Gamma_i$, we consider the two-dimensional cross section
consisting of the measures
\begin{equation*}
    \Phi_i(\eta,s,t) \coloneqq  t\delta_0+(1-t)\bigl(s\mu^{*}+(1-s)\eta\bigr)\text{,} \qquad (s,t) \in [0,1]^2\text{.}
\end{equation*}
For any $\eta \in \Gamma_i$, we see that $\Phi_i(\eta,0,0) = \eta$, $\Phi_i(\eta,1,0) = \mu^{*}$, and $\Phi_i(\eta,0,1) = \Phi_i(\eta,1,1) = \delta_0$.
Thus these cross sections foliate the region $\left\{\mu_t \;\vert\; \mu\in R_i\;\text{and}\; t\in[0,1]\right\}$, up to null sets.

\begin{figure}[H]
    \centering
    \begin{tikzpicture}[scale=4, every node/.style={font=\small}]
    
      \pgfmathsetmacro{\h}{sqrt(3)/2}     
      \coordinate (A) at (0,0);           
      \coordinate (B) at (1,0);           
      \coordinate (C) at (0.5,\h);        
    
      \coordinate (G) at (0.5,\h/3);
    
      \coordinate (M) at (0.35,0);

      \fill[gray!20] (A)--(B)--(G)--cycle;
    
      \draw[line width=1.4pt] (A) -- node[midway, below] (gThreeLabel) {$\Gamma_3$} (B);
      \draw[line width=0.7pt] (B) -- node[midway, above right] (gOneLabel) {$\Gamma_1$} (C);
      \draw[line width=0.7pt] (C) -- node[midway, above left] (gTwoLabel) {$\Gamma_2$} (A);
    
      \draw (A)--(G);
      \draw (B)--(G);
      \draw (C)--(G);
    
      \draw[dotted, thick] (M)--(G);
    
      \fill (G) circle (0.03) node[above left=1pt] {$\mu^*$};
      \fill (M) circle (0.025) node[below=2pt] {$\eta$};
    
      \fill (A) circle (0.025) node[below left=2pt] {$\delta_1$};
      \fill (B) circle (0.025) node[below right=2pt] {$\delta_2$};
      \fill (C) circle (0.025) node[above=2pt] {$\delta_3$};
    
    \end{tikzpicture}
    \caption{The region $\Gamma$ when $q = 4$.}
    \label{fig:q=4}
\end{figure}
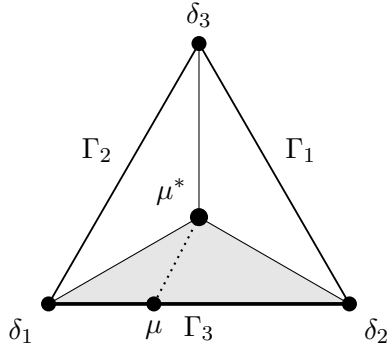

\begin{proposition}
    Let $f \colon [q]^n \to \{0,1\}$ be a $0$-monotone and symmetric function. There exists a constant $C(q)$ such that for any region $R_i$, we have
    \begin{equation*}
        \gamma\left(\{\mu_t\;\vert\;\mu \in R_i\;\text{and}\;t\in[0,1]\} \cap \mathscr{D}\right) \leq C(q)\, \frac{\ln(1-\varepsilon)-\ln(\varepsilon)}{\log n}\text{,}
    \end{equation*}
\end{proposition}
\begin{proof}
    Throughout the proof, $C = C(q)$ denotes different constants at different lines, depending on $q$ only.
    For $\eta \in \Gamma_i$ and $(s,t)\in[0,1]^2$, we define
    \begin{equation*}
        \Phi_i(\eta,s,t)
        \coloneqq
        t\delta_0+(1-t)\bigl(s\mu^{*}+(1-s)\eta\bigr)\text{.}
    \end{equation*}
    For $s>0$, the measure $s\mu^{*}+(1-s)\eta$ belongs to $R_i$, and its $i$th atom is
    \begin{equation*}
        s\mu^{*}(i)+(1-s)\eta(i)=\frac{s}{q-1}\text{.}
    \end{equation*}
    Therefore, by Proposition~\ref{prop:length-of-line}, for every fixed $\eta\in\Gamma_i$ and every $s>0$,
    \begin{equation*}
        \len\left(\left\{t\in[0,1]\;\vert\;\Phi_i(\eta,s,t)\in\mathscr{D}\right\}\right)
        \leq
        C\, \frac{(\ln (1/2) - \ln (\varepsilon))\log((q-1)/s)}{\log n}\text{.}
    \end{equation*}
    The slice $s=0$ is irrelevant for the following integration. We now integrate this estimate over all cross sections. The map $ \Phi_i\colon\Gamma_i\times[0,1]^2\to \Delta[q]$ parametrizes $\{\mu_t\;\vert\; \mu\in R_i\;\text{and}\;t\in[0,1]\}$ up to null sets. Moreover, in affine coordinates its Jacobian is of the form
    \begin{equation*}
        c_q(1-t)^{q-2}(1-s)^{q-3}\text{,}
    \end{equation*}
    and hence the density of the pullback of $\gamma$ under this parametrization is bounded above by a constant depending only on $q$. Consequently, for every measurable set $B$ contained in the image of this parametrization,
    \begin{equation*}
        \gamma(B)
        \leq
        C(q)
        \int_{\Gamma_i}\int_{0}^{1}\int_{0}^{1}
        \mathbb{I}[\Phi_i(\eta,s,t)\in B]\,dt\,ds\,d\eta\text{,}
    \end{equation*}
    where $d\eta$ denotes $(q-3)$-dimensional Lebesgue measure on $\Gamma_i$. Applying this with
    \begin{equation*}
        B=
        \left\{\mu_t\;\vert\;\mu\in R_i\;\text{and}\;t\in[0,1]\right\}\cap\mathscr{D}\text{,}
    \end{equation*}
    and then using the preceding line-segment bound, we get
    \begin{align*}
        \gamma(B)
        &\leq
        C(q)
        \int_{\Gamma_i}\int_0^1
        \len\left(\left\{t\in[0,1]\;\vert\;\Phi_i(\eta,s,t)\in\mathscr{D}\right\}\right)
        \,ds\,d\eta \\
        &\leq
        C(q)\, \frac{\ln (1/2) - \ln(\varepsilon)}{\log n}
        \int_{\Gamma_i}\int_0^1
        \log\left(\frac{q-1}{s}\right)
        \,ds\,d\eta\text{.}
    \end{align*}
    Finally, we note that
    \begin{equation*}
        \int_0^1 \log\left(\frac{q-1}{s}\right)\,ds
        =
        \log(q-1)+\frac{1}{\ln2}
        \leq C(q)\text{,}
    \end{equation*}
    and the total $(q-3)$-dimensional volume of $\Gamma_i$ is also bounded by a constant depending only on $q$. Therefore
    \begin{equation*}
        \gamma\left(
        \left\{\mu_t\;\vert\;\mu\in R_i\;\text{and}\;t\in[0,1]\right\}
        \cap \mathscr{D}
        \right)
        \leq
        C(q)\frac{\ln(1/2)-\ln(\varepsilon)}{\log n} \leq C(q)\frac{\ln(1-\varepsilon)-\ln(\varepsilon)}{\log n}\text{.}
    \end{equation*}
    This proves the proposition.
\end{proof}

\subsection{Tightness}

We remark that Theorem \ref{thm:main} is sharp up to multiplicative constants. Ben-Or and Linial \cite{Ben-OrLinial1990} constructed the tribes function that showed that the $O(\log n/n)$ bound on influences presented by Kahn, Kalai, and Linial \cite{KKL} is sharp. We present a generalized tribes function. Fix $q \geq 2$, and let
\begin{equation*}
    \rho = \frac{2}{q+1}\text{.}
\end{equation*}
Then $\rho > 1/q$. Let $r$ be some positive integer, let $m = \left\lfloor \rho^{-r}\right\rfloor$, and set $n = mr$. It is enough to construct examples for $n$ of this form. Partition the set $\{1,2,\ldots,n\}$ into disjoint sets $T_1,T_2,\ldots,T_m$ of equal size $r$. For $a\in[q]$ and $x\in[q]^n$, define
\begin{equation*}
    N_a(x)=\#\{i\in\{1,\ldots,m\}\;\vert\;x_j=a\text{ for all }j\in T_i\}\text{.}
\end{equation*}
We define $f \colon [q]^n \to [q]$ by
\begin{equation*}
    f(x)=\min\arg\max_{a\in[q]} N_a(x)\text{,}
\end{equation*}
where the minimum is taken with respect to the order $0<1<\cdots<q-1$.

Note that the function $f$ is symmetric in the sense of Definition \ref{def:symmetric}. Indeed, $f$ is invariant under permutations of the blocks and under permutations of the coordinates inside each block, and this group acts transitively on $\{1,\ldots,n\}$. The function $f$ is also monotone. To see this, suppose $x\leq_a y$. Then $ N_a(x)\leq N_a(y)$ and $ N_b(y)\leq N_b(x)$ for every $b\neq a$. Thus, if $f(x)=a$, then $a$ remains a maximizer after passing from $x$ to $y$, and the tie-breaking order still selects $a$. Hence $f(y)=a$.

We now show that this construction gives the correct lower bound. Fix $0<\varepsilon<1/2$, and choose constants $c_-<c_+$ such that
\begin{equation*}
    -\ln(1-\varepsilon)<c_-<c_+<-\ln(\varepsilon)\text{.}
\end{equation*}
For each $r$, define $p_\pm \coloneqq p_\pm(r)$ by
\begin{equation*}
    p_-\coloneqq\left(\frac{c_-}{m}\right)^{1/r}
    \quad\text{and}\quad
    p_+\coloneqq\left(\frac{c_+}{m}\right)^{1/r}\text{.}
\end{equation*}
Since $m=\lfloor\rho^{-r}\rfloor$, we have $p_+-p_-=\log((1-\varepsilon)/\varepsilon)\cdot\Theta(1/r)$, provided that $c_+$ and $c_-$ are sufficiently close to $-\log(\varepsilon)$ and $-\log(1-\varepsilon)$, respectively.  Moreover, $r=\Theta(\log n)$ as $n=mr$. Thus,
\begin{equation*}
    p_+-p_-=\log((1-\varepsilon)/\varepsilon)\cdot \Theta(1/\log n)\text{.}
\end{equation*}
Choose $\delta>0$ small enough so that
\begin{equation*}
    \frac{1-\rho}{q-1}<\rho-2\delta\text{.}
\end{equation*}
This is possible because $\rho>1/q$. First consider $a\in[q]\setminus\{0\}$. Let
\begin{equation*}
    R_a=
    \{\mu\in\Delta[q]\;\vert\;p_-\leq \mu(a)\leq p_+
    \text{ and }\mu(b)\leq \rho-\delta\text{ for all }b\neq a\}\text{.}
\end{equation*}
By the choice of $\delta$, the sections of $R_a$ obtained by fixing $\mu(a)$ have volume bounded above and below by positive constants depending only on $q$. Hence
\begin{equation*}
    \gamma(R_a)=\Theta (p_+-p_-)
    =\log((1-\varepsilon)/\varepsilon) \cdot\Theta(1/\log n)\text{.}
\end{equation*}
For $\mu\in R_a$, set $ \lambda_a=m\mu(a)^r.$ Then $\lambda_a\in[c_-,c_+]$. On the other hand, for $b\neq a$,
\begin{equation*}
    \mathbb{P}_{x\sim\mu^{\otimes n}}[N_b(x) \geq 1] \leq \mathbb{E}_{x\sim\mu^{\otimes n}}[N_b(x)]
    =
    m\mu(b)^r
    \leq m(\rho-\delta)^r
    =o(1)\text{.}
\end{equation*}
By a union bound, no color $b\neq a$ has an all-$b$ block with probability $1-o(1)$, so
\begin{align*}
    \mathbb{P}_{x\sim\mu^{\otimes n}}[f(x)=a]
    &=
    \mathbb{P}_{x\sim\mu^{\otimes n}}[N_a(x)\geq1]+o(1) \\
    &=
    1-(1-\mu(a)^r)^m+o(1) \\
    &=
    1-\exp(-\lambda_a)+o(1)\text{.}
\end{align*}
By the choice of $c_-$ and $c_+$, it follows that for all sufficiently large $r$,
\begin{equation*}
    \varepsilon
    \leq
    \mathbb{P}_{x\sim\mu^{\otimes n}}[f(x)=a]
    \leq
    1-\varepsilon
\end{equation*}
for every $\mu\in R_a$. Consequently,
\begin{equation*}
    \gamma(\{\mu \in \Delta[q]\;\vert\;\varepsilon \leq
    \mathbb{P}_{x\sim\mu^{\otimes n}}[f(x)=a]\leq 1-\varepsilon\})
    \geq
    \log((1-\varepsilon)/\varepsilon) \cdot\Theta (1/\log n)
\end{equation*}
for every $a\in[q]\setminus\{0\}$.

It remains to prove the color $a=0$ since the deterministic tie-breaking rule favors $0$ when no color has a monochromatic block. In this case we use the color $a=1$ as the competing color. Define
\begin{equation*}
    R_0= \{\mu\in\Delta[q]\;\vert\;p_-\leq \mu(1)\leq p_+\;\text{and}\;\mu(b)\leq \rho-\delta\;\text{for all}\;b\neq 1\}\text{.}
\end{equation*}
As before, $ \gamma(R_0)=\log((1-\varepsilon)/\varepsilon) \cdot\Theta(1/\log n)$. For $\mu\in R_0$, set $ \lambda_1=m\mu(1)^r \in [c_-,c_+]$, and note
\begin{equation*}
    \mathbb{E}_{x\sim\mu^{\otimes n}}[N_b(x)]=o(1)
\end{equation*}
for every $b\neq 1$. Thus, with probability $1-o(1)$, the only possible all-monochromatic blocks have color $1$. Hence
\begin{align*}
    \mathbb{P}_{x\sim\mu^{\otimes n}}[f(x)=0] &=
    \mathbb{P}_{x\sim\mu^{\otimes n}}[N_1(x)=0]+o(1) \\
    &= (1-\mu(1)^r)^m+o(1) \\
    &= \exp(-\lambda_1)+o(1)\text{.}
\end{align*}
Again, by the choice of $c_-$ and $c_+$, for all sufficiently large $r$,
\begin{equation*}
    \varepsilon \leq \mathbb{P}_{x\sim\mu^{\otimes n}}[f(x)=0] \leq 1-\varepsilon
\end{equation*}
for every $\mu\in R_0$. Combining the two cases, for every $a\in[q]$ we obtain
\begin{equation*}
    \gamma(\{\mu \in \Delta[q]\;\vert\;\varepsilon \leq
    \mathbb{P}_{x\sim\mu^{\otimes n}}[f(x)=a]\leq 1-\varepsilon\})
    =
    \log((1-\varepsilon)/\varepsilon) \cdot\Theta(1/\log n)\text{,}
\end{equation*}
where the upper bound follows from Theorem \ref{thm:main}. This proves that the logarithmic dependence in Theorem \ref{thm:main} is sharp.

\bibliographystyle{plain}

\end{document}